\documentclass[11pt]{amsart}

\usepackage{epigamath}

\usepackage[english]{babel}

\usepackage{amscd, amsmath, amssymb, amsthm}
\usepackage[utf8]{inputenc}
\usepackage[all,cmtip]{xy}

\usepackage{tikz}
\usepackage{pgffor}
\usetikzlibrary{arrows,backgrounds}
\usetikzlibrary{calc}

\newtheorem{theorem}{Theorem}[section]
\newtheorem{corollary}[theorem]{Corollary}
\newtheorem{proposition}[theorem]{Proposition}
\newtheorem{lemma}[theorem]{Lemma}

\theoremstyle{definition}

\newtheorem{conjecture}[theorem]{Conjecture}

\theoremstyle{remark}
\newtheorem{remark}[theorem]{Remark}

\def\Z{\text{\bf Z}}
\def\Q{\text{\bf Q}}
\def\R{\text{\bf R}}
\def\C{\text{\bf C}}
\def\P{\text{\bf P}}

\def\arrow{\rightarrow}
\def\inj{\hookrightarrow}

\DeclareMathOperator{\Hom}{Hom}

\DeclareMathOperator{\Hilb}{\mathrm{Hilb}}

\def\Sm{\mathrm{Sm}}

\def\Gr{\mathrm{Gr}}

\DeclareMathOperator{\ord}{ord}
\DeclareMathOperator{\inter}{int}
\DeclareMathOperator{\Sper}{Sper}
\DeclareMathOperator{\Spc}{Spc}

\DeclareMathOperator{\BGL}{BGL}
\DeclareMathOperator{\GL}{GL}

\EpigaVolumeYear{5}{2021}
\EpigaArticleNr{12}
\ReceivedOn{September 22, 2020}
\InFinalFormOn{May 25, 2021}
\AcceptedOn{June 23, 2021}

\title{Torus actions, Morse homology,
and the Hilbert scheme of points on affine space}
\titlemark{Torus actions, Morse homology and the Hilbert scheme of points on affine space}

\author{Burt Totaro}
\address{UCLA Mathematics Department, Box 951555,
Los Angeles, CA 90095-1555}
\email{totaro@math.ucla.edu}

\authormark{B. Totaro}

\AbstractInEnglish{We formulate a conjecture on actions of the multiplicative group in motivic homotopy theory. In short, if the multiplicative group $\mathbf{G}_m$ acts on a quasi-projective scheme $U$ such that $U$ is attracted as $t$ approaches $0$ in $\mathbf{G}_m$ to a closed subset $Y$ in $U$, then the inclusion from $Y$ to $U$ should be an $\mathbf{A}^1$-homotopy equivalence.
We prove several partial results. In particular, over the complex numbers, the inclusion is a homotopy equivalence on complex points. The proofs use an analog of Morse theory for singular varieties. Application: the Hilbert scheme of points on affine $n$-space is homotopy equivalent to the subspace consisting of schemes supported at the origin. }



\TitleInFrench{Actions du tore, homologie de Morse et schéma de Hilbert des points de l'espace affine}

\AbstractInFrench{Nous formulons une conjecture sur les actions du groupe multiplicatif en théorie de l'homotopie motivique. En bref, si le groupe multiplicatif $\mathbf{G}_m$ opère sur un schéma quasi-projectif $U$, de sorte que lorsque $t$ tend vers $0$ dans $\mathbf{G}_m$, $U$ est contracté sur un sous-ensemble fermé $Y$, alors l'inclusion de $Y$ dans $U$ devrait être une $\mathbf{A}^1$-équivalence d'homotopie.
Nous démontrons plusieurs résultats partiels. En particulier, sur le corps des complexes, l'inclusion est une équivalence d'homotopie sur les points complexes. Les preuves utilisent des analogues de la théorie de Morse pour des variétés singulières. Application : le schéma de Hilbert des points de l'espace affine de dimension $n$ est homotopiquement équivalent au sous-espace dont les points sont les schémas supportés à l'origine.
}

\acknowledgement{This work was supported by NSF grant DMS-1701237.
I spoke about it in the Algebraic Geometry and Moduli Zoominar
at ETH Zurich.}

\begin{document}


\removeabove{0,3cm}
\removebetween{0,3cm}
\removebelow{0,3cm}

\maketitle

\begin{prelims}

\DisplayAbstractInEnglish

\bigskip

\DisplayKeyWords

\medskip

\DisplayMSCclass

\bigskip

\languagesection{Fran\c{c}ais}

\bigskip

\DisplayTitleInFrench

\medskip

\DisplayAbstractInFrench

\end{prelims}


\newpage

\setcounter{tocdepth}{1} 

\tableofcontents


\section{Introduction}

We formulate a conjecture on actions of the multiplicative group
$\mathbf{G}_m$ in algebraic geometry. (Over the complex numbers,
this group may be called $\C^*$.) In short, if $\mathbf{G}_m$ acts
on a quasi-projective scheme $U$ which is
attracted as $t\to 0$ in $\mathbf{G}_m$
to a closed subset $Y$ in $U$, then the inclusion $Y\to U$ should be
an $\mathbf{A}^1$-homotopy equivalence (Conjecture \ref{conj:torus}).
This is not obvious, in that
the action of $\mathbf{G}_m$ on $U$
usually does not extend to a morphism $\mathbf{A}^1\times U
\arrow U$; compare Figure \ref{torusaction}.
We show that the inclusion $Y\to U$ over the complex numbers
is at least a homotopy equivalence in the classical topology
(Theorem \ref{complex}). This extends work of Hausel and Rodriguez-Villegas
on the case where $U$ is smooth \cite[Corollary 1.3.6]{HR}.
We prove several other results
in the direction of the conjecture, including
a homotopy equivalence on real points (Theorem \ref{real})
and, when $U$ is smooth,
an $\mathbf{A}^1$-homotopy equivalence after a suitable suspension (Theorem
\ref{smooth}). The proofs use the ideas
of Morse homology, translated into algebraic geometry
(Proposition \ref{broken}).
\begin{figure}[b]\centering
\begin{tikzpicture}[xscale=1, yscale=1, >=stealth']
\draw (-2,0) -- (-.5,0);
\draw[<->] (-.5,0) -- (0.5,0);
\draw (0.5,0) -- (2,0);
\draw[->] (3,0) -- (2,0);
\draw[->] (0,2) -- (0,1);
\draw (0,1) -- (0,-1);
\draw[->] (0,-2) -- (0,-1);
\draw[->] (1,2) -- (1,1);
\draw (1,1) -- (1,-1);
\draw[->] (1,-2) -- (1,-1);
\draw[domain=-2:-1,variable=\x] plot({\x},{(-0.5)*(\x - 1)/(\x)});
\draw[<-][domain=-1:-0.3333,variable=\x] plot({\x},{(-0.5)*(\x - 1)/(\x)});
\draw[domain=-2:-1,variable=\x] plot({\x},{(0.5)*(\x - 1)/(\x)});
\draw[<-][domain=-1:-0.3333,variable=\x] plot({\x},{(0.5)*(\x - 1)/(\x)});
\draw[->][domain=0.5:8/11,variable=\x] plot({\x},{(2)*(\x - 1)/(\x)});
\draw[domain=8/11:1.5,variable=\x] plot({\x},{(2)*(\x - 1)/(\x)});
\draw[->][domain=3:1.5,variable=\x] plot({\x},{(2)*(\x - 1)/(\x)});
\draw[->][domain=0.5:8/11,variable=\x] plot({\x},{(-2)*(\x - 1)/(\x)});
\draw[->][domain=3:1.5,variable=\x] plot({\x},{(-2)*(\x - 1)/(\x)});
\draw[domain=8/11:1.5,variable=\x] plot({\x},{(-2)*(\x - 1)/(\x)});
\draw[domain=-2:-1/2,variable=\x] plot({\x},{(-1/8)*(\x - 1)/(\x)});
\draw[<-][domain=-1/2:-1/15,variable=\x] plot({\x},{(-1/8)*(\x - 1)/(\x)});
\draw[domain=-2:-1/2,variable=\x] plot({\x},{(1/8)*(\x - 1)/(\x)});
\draw[<-][domain=-1/2:-1/15,variable=\x] plot({\x},{(1/8)*(\x - 1)/(\x)});
\draw[->][domain=1/9:0.4,variable=\x] plot({\x},{(-1/4)*(\x - 1)/(\x)});
\draw[->][domain=3:1.9,variable=\x] plot({\x},{(-1/4)*(\x - 1)/(\x)});
\draw[domain=0.4:1.9,variable=\x] plot({\x},{(-1/4)*(\x - 1)/(\x)});
\draw[->][domain=1/9:0.4,variable=\x] plot({\x},{(1/4)*(\x - 1)/(\x)});
\draw[->][domain=3:1.9,variable=\x] plot({\x},{(1/4)*(\x - 1)/(\x)});
\draw[domain=0.4:1.9,variable=\x] plot({\x},{(1/4)*(\x - 1)/(\x)});
\node[left] at (-2,0) {Y};
\node[left] at (-2,1.5) {U};
\draw[fill] (0,0) circle[radius=0.05];
\draw[fill] (1,0) circle[radius=0.05];
\end{tikzpicture}
\caption{Example of a $T$-action on $U\cong \P^1\times \mathbf{A}^1$,
$t([x_0,x_1],y)=([x_0,tx_1],ty)$,
with $Y=\P^1\times 0$ shown
as the horizontal line ($T=\mathbf{G}_m$). The arrows point
in the direction $t\to 0$. The fixed point set $Y^T$
consists of two points.}
\label{torusaction}
\end{figure}
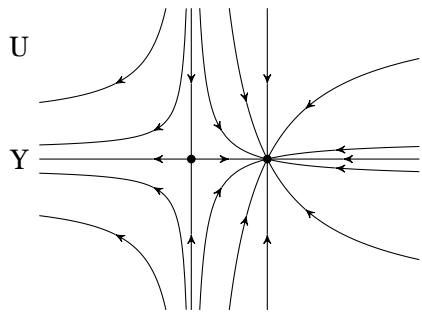

We apply these results to the Hilbert scheme of points
on affine space.
The Hilbert scheme of points on an algebraic surface is smooth,
and its Betti numbers were computed by G\"ottsche 
\cite{GoettscheICM}. The Hilbert scheme of points on a higher-dimensional
variety, even affine space $\mathbf{A}^n$,
is more mysterious. It has many irreducible components
\cite{Iarrobino},
and for $n\geq 16$ its singularities satisfy Murphy's law up to retraction
\cite{Jelisiejew}. Nonetheless, progress was recently made
toward understanding the homotopy type
(and even the $\mathbf{A}^1$-homotopy type) of $\Hilb_d(\mathbf{A}^n)$ for $n$ large
compared to $d$. In particular, in the limit where $n$ goes
to infinity, $\Hilb_d(\mathbf{A}^{\infty})$ has the $\mathbf{A}^1$-homotopy type
of the infinite
Grassmannian $\Gr_{d-1}(\mathbf{A}^{\infty})\simeq \BGL(d-1)$ \cite{HJNTY}.
There are also corresponding stability theorems.
In particular, over the complex numbers, the resulting homomorphism
on integral cohomology,
$$H^*(\BGL(d-1,\C),\Z)=\Z[c_1,\ldots,c_{d-1}]\arrow H^*(\Hilb_d(\mathbf{A}^n),\Z),$$
is an isomorphism in degrees at most $2n-2d+2$ \cite{HJNTY}.

This paper considers another homotopical property
of the Hilbert scheme $\Hilb_d(\mathbf{A}^n)$ for finite $n$.
Namely, over the complex numbers, we show that
$\Hilb_d(\mathbf{A}^n)$ (in the classical topology)
has the homotopy type of $\Hilb_d(\mathbf{A}^n,0)$, the (compact) subspace
of schemes supported at the origin (Corollary
\ref{move-to-origin}).
This result is deduced from Theorem
\ref{complex} on the topology of $\mathbf{G}_m$-actions.
For example, it follows that the weight filtration on the rational
cohomology $H^i(\Hilb_d(\mathbf{A}^n),\Q)$ is concentrated in weights $\leq i$,
since that holds for proper schemes over $\C$ \cite{Deligne}.

It remains open whether $\Hilb_d(\mathbf{A}^n,0)$
is $\mathbf{A}^1$-homotopy equivalent to $\Hilb_d(\mathbf{A}^n)$, over the complex numbers
or any other field. This would follow from our general
Conjecture \ref{conj:torus}. We can say something
about the unstable $\mathbf{A}^1$-homotopy type of these spaces,
namely that $\Hilb_d(\mathbf{A}^n,0)$ and $\Hilb_d(\mathbf{A}^n)$
are $\mathbf{A}^1$-connected (Theorems \ref{hilbconnected}
and \ref{hilbzeroconnected}).

As a tool, we extend one of Bachmann's conservativity
theorems, relating the motivic stable homotopy category
to the derived category of motives along with real realizations
(Theorem \ref{conservativity}).

\subsection*{Acknowledgments}
Thanks to Tom Bachmann, David Hemminger, Marc Hoyois, Joachim Jelisiejew,
Denis Nardin, Maria Yakerson, and the referee
for their suggestions.

\section{Main results,
and a conjecture on $\mathbf{G}_m$-actions in motivic homotopy theory}
\label{section:main}

In this section, we formulate
a general conjecture about actions of the multiplicative
group $T=\mathbf{G}_m$ in motivic homotopy theory.
(For motivic homotopy theory as defined by Morel and Voevodsky,
a reference is \cite{MV}
and an introduction is \cite{AE}.) Roughly, if $T$ acts
on a quasi-projective scheme $U$ which is
attracted as $t\to 0$ in $T$
to a closed subset $Y$ in $U$, then the inclusion $Y\to U$ should be
an $\mathbf{A}^1$-homotopy equivalence (Conjecture \ref{conj:torus}).
We show that (over the complex numbers) the inclusion $Y\to U$
is at least a homotopy equivalence in the classical topology
(Theorem \ref{complex}, proved in Section \ref{section:proof}).

Let $\Hilb_d(\mathbf{A}^n)$ be the quasi-projective scheme
of zero-dimensional degree-$d$ closed subschemes
of affine space $\mathbf{A}^n$ over a field $k$.
When $k$ is the complex numbers, we deduce from Theorem \ref{complex}
that $\Hilb_d(\mathbf{A}^n)$
has the homotopy type of $\Hilb_d(\mathbf{A}^n,0)$, the (compact) subspace
of schemes supported at the origin (Corollary
\ref{move-to-origin}).

Here is our general conjecture on actions of the multiplicative
group.
Let $X$ be a projective scheme over a field $k$
with an action of $T=\mathbf{G}_m$.
Suppose that there is a $T$-equivariant ample line bundle
on $X$. Let $Y$
be a $T$-invariant closed subset of $X$ 
such that every point $x$ in $X$ with
$\lim_{t \to \infty}(tx)\in Y$ is in $Y$. Suppose
that the fixed point set $Y^T$ is open in $X^T$.
Let $U$ be the subset of points $x$ in $X$
such that $\lim_{t \to 0}(tx)$ is in $Y$. 
We show in Lemma \ref{lemma:open} that $Y$ is contained in $U$ and
$U$ is open in $X$. 

\begin{conjecture}\label{conj:torus}
The inclusion $Y\to U$ is an $\mathbf{A}^1$-homotopy equivalence (that is,
an isomorphism in the $\mathbf{A}^1$-homotopy category $H(k)$).
\end{conjecture}

The assumption that $X$ has a $T$-equivariant ample
line bundle is automatic if $X$ is normal
\cite[Theorem 1.6]{Sumihiro}.

Over the complex numbers, the proof
of Theorem \ref{complex} shows that $Y$ is an
{\it attracting set }for the $T$-action on $X$,
in the terminology of topological dynamics,
and $U$ is the {\it basin of attraction }for $Y$
\cite[section 1]{Milnor}. To say that $Y$ is an attracting set means
that there is a neighborhood $N_1$ of $Y$ (in the classical topology)
for which the images
$t(N_1)$ converge to $Y$, meaning that for every
neighborhood $N_2$ of $Y$, there is an $r>0$ such that
$t(N_1)\subset N_2$ for all $t\in \C^*$ with $|t|<r$.

The conjecture would be useful for motivic homotopy
theory, since $\mathbf{G}_m$-actions occur everywhere. When $U$ is smooth,
both $Y$ and $U$ are unions of affine bundles over the connected components
of $Y^T$, by Bia\l ynicki-Birula \cite{BB}. But even then,
we only know how to prove that $Y\to U$ is an $\mathbf{A}^1$-homotopy equivalence
after a suitable suspension (Theorem \ref{smooth}).
Regardless of whether $U$ is smooth,
the conjecture would be clear if the $T$-action on $U$
extended to a morphism
$$\mathbf{A}^1\times U\to U,$$
since $0\times U$ would map into $Y$;
but in general there is no such morphism.
Even the $T$-action on $Y$ need not extend
to a morphism $\mathbf{A}^1\times Y\to Y$:
consider the case where $Y$ is $\P^1$ with the
standard action of $T$, where $\lim_{t\to 0}tx=0$ if $x\neq\infty$
but $\lim_{t\to 0}t(\infty)=\infty$.

Another way to describe the same situation is Drinfeld's
analog of the Bia\l ynicki-Birula decomposition
for singular varieties, although we will
not use that explicitly in what follows. Namely, Drinfeld
defines an algebraic space $Y^+$, the ``attractor'' of $Y$,
as the space of $T$-equivariant morphisms
$\mathbf{A}^1\to Y$; roughly speaking,
a point of $Y^+$ is a point $x$ of $Y$ together
with a limit point $\lim_{t\to 0}tx$. Drinfeld shows
that $Y^+\to Y$ is bijective for $Y$ proper over $k$,
although usually not an isomorphism \cite[Proposition 1.4.11]{Drinfeld}.
For example, if $Y=\P^1$ with the standard $T$-action,
then the space $Y^+$ is the disjoint union of $\mathbf{A}^1$ and the point
at infinity. In a sense, the difficulty for Conjecture \ref{conj:torus}
is that the action of $T$ on $U$ does not extend to an action
of the multiplicative monoid $\mathbf{A}^1$. The action of $T$ on $U^+$ does
extend to an action of $\mathbf{A}^1$; but that does not obviously help,
because the morphism $U^+\to U$ is usually not a homotopy equivalence.

As evidence for Conjecture \ref{conj:torus} in the singular case,
we prove the following weaker statement in section
\ref{section:proof}. Theorem \ref{complex}
was proved in the case where $U$ is smooth by Hausel
and Rodriguez-Villegas \cite[Corollary~1.3.6]{HR}.

\begin{theorem}\label{complex}
Under the assumptions of Conjecture \ref{conj:torus}
with base field $\C$,
the inclusion $Y\to U$ is a homotopy equivalence
$($in the classical topology$)$.
\end{theorem}

\begin{corollary}\label{move-to-origin}
Over the complex numbers, the inclusion
from $\Hilb_d(\mathbf{A}^n,0)$ to $\Hilb_d(\mathbf{A}^n)$ $($in the classical topology$)$
is a homotopy equivalence.
\end{corollary}

\begin{proof}[Proof of Corollary \ref{move-to-origin}]
Let $X=\Hilb_d(\P^n)$ and $Y=\Hilb_d(\mathbf{A}^n,0)$.
The idea is to use the action of the multiplicative
group $T$ (that is, $\C^*$) on $\Hilb_d(\P^n)$,
coming from the action
of $T$ on $\P^n$ by $$t([x_0,\ldots,x_n])=[x_0,tx_1,\ldots,tx_n].$$
(We identify $\mathbf{A}^n$ with the open subset $x_0\neq 0$ in $\P^n$.)
Here $X$ has a
$\GL(n+1)$-equivariant ample line bundle by construction.
(Namely, Grothendieck constructed the Hilbert scheme
as a closed subscheme of the Grassmannian
of subspaces of the vector space
of homogeneous polynomials of sufficiently high
degree, sending a closed subscheme $S\subset \P^n$ to the linear subspace
of polynomials that vanish on $S$ \cite[Section~I.1]{Kollarrat}.
The standard ample line bundle $O(1)$ on the Grassmannian
is $GL(n+1)$-equivariant.)
In particular, $X$ has
a $T$-equivariant ample line bundle.

The open subset $U\subset X$ of 0-dimensional schemes
that converge as $t\in T$ approaches 0 to a subscheme supported
at $[1,0,\ldots,0]$ is exactly $\Hilb_d(\mathbf{A}^n)$.
The action of $T$ on $U$ need not extend to a morphism $\mathbf{A}^1\times U\to U$
(or even $\mathbf{A}^1\times Y\to Y$). Nonetheless, the desired
homotopy equivalence follows from Theorem \ref{complex}.
\end{proof}

\begin{remark}
The action of $T$ on $Y=\Hilb_d(\mathbf{A}^n,0)$ does not extend to a morphism
$\mathbf{A}^1\times Y\to Y$ in any case where $Y\neq Y^T$.
For example, for $Y=\Hilb_3(A^2,0)$, the point $S_a:=\{ x=ay^2, y^3=0\}$
in $Y$ has $\lim_{t\to 0}t(S_a)=
Z:=\{ x^2=0,xy=0,y^2=0\}$ for any $a\neq 0\in \C$, whereas
the point $S_0:=\{x=0, y^3=0\}$ in $Y$ is fixed by $T$
and hence has $\lim_{t\to 0}t(S_0)=S_0$.
\end{remark}

\section{$\mathbf{G}_m$-actions and broken trajectories}

We show here that for an action of the multiplicative group $T=\mathbf{G}_m$
on a projective scheme,
every limit of $T$-orbits is a {\it broken trajectory},
meaning a chain of $T$-orbits that connect a finite sequence
of $T$-fixed points. This is analogous to fundamental results
in Morse homology. Namely, given a smooth function
on a closed Riemannian manifold satisfying some mild conditions,
every limit of gradient flow lines is a broken trajectory,
meaning a chain of gradient flow lines that connect
a finite sequence of critical points
\cite[Theorem 4.9, Definition 4.10]{BH}.
For a $T$-action on a smooth complex projective variety, one can
in fact deduce the results here from those in Morse homology, applied
to a Hamiltonian function for the $T$-action. Instead,
we give a direct proof over any field. It turns out that
smoothness is irrelevant.

\begin{proposition}
\label{broken}
Let $X$ be a projective scheme over a field $k$
with an action of $T=\mathbf{G}_m$.
Suppose that there is a $T$-equivariant ample line bundle
on $X$. Then every limit of $T$-orbit closures in $X$ $($in the Chow variety
of effective 1-cycles on $X)$ is a broken trajectory, that is,
a chain of $T$-orbits $($with some positive multiplicities$)$
connecting some $T$-fixed points.

In more detail: let $C$ be a smooth curve over $k$ with a morphism
$f\colon C\to X$, not mapping into $X^T$.
Composing $f_T\colon T\times C\to T\times X$
with the action of $T$
gives a morphism $T\times C\to X$, which extends to a morphism
$\P^1\times (C-Z)\to X$ for some 0-dimensional closed subset $Z$ in $C$.
This gives a morphism $e$ from $C-Z$ to the Chow variety of 1-cycles on $X$,
which extends to all of $C$ by properness of Chow varieties. Then for
each $k$-point $c$ in $C$ $($possibly in $Z)$, $e(c)$ is a broken trajectory
over $k$, meaning the sum of $T$-orbits $($with some positive
multiplicities$)$ of points
$y_1,\ldots,y_n$ in $X(k)$ that connect $T$-fixed points
$x_0,\ldots,x_n$ in $X(k)$. More precisely, $\lim_{t\to 0}t(y_i)=x_{i-1}$
and $\lim_{t\to \infty}t(y_i)=x_i$ for each $1\leq i\leq n$.
\end{proposition}

If the image of $f\colon C\to X$ is contained
in $X^T$, then the morphism $e$ to the Chow variety of 1-cycles
is constant (equal to zero as a 1-cycle).
The proposition would still be true in that case
if suitably interpreted: namely, any limit
of $T$-fixed points in $X$ is a $T$-fixed point.

\begin{proof}
There is a $T$-equivariant embedding of $X$ into the projective
space $P(V)$ for some representation $V$ of $T$. Given that,
we can assume that $X=P(V)$; this greatly simplifies the situation.
Then $T$ acts on $X=\P^r$ by
$t([z_0,\ldots,z_r])=[t^{a_0}z_0,\ldots,t^{a_r}z_r]$
for some integers $a_i$.
We can assume that $a_0\leq\cdots\leq a_r$.

Composing $f\colon C\to X$ with the action of $T$ on $X$
gives a morphism $T\times C\to X$, which can be viewed
as a rational map $G\colon \P^1\times C\dashrightarrow X$ over $k$.
Since $X$ is proper over $k$,
$G$ becomes a morphism $W\to X$, where $W$ is a surface obtained by
blowing up $\P^1\times C$ finitely many times
at closed points. In particular, $G$ restricts to a morphism $\P^1\times (C-Z)
\to X$ for some 0-dimensional closed subset $Z$ of $X$.

Because $C$ is normal and all fibers of $W\to C$
have dimension 1, the fibers of $W\to C$ form
a well-defined family of effective
1-cycles on $W$, and hence they give a morphism from $C$ to the Chow
variety of 1-cycles on $W$ \cite[Theorem I.3.17]{Kollarrat}. Pushing
cycles forward makes the Chow
variety covariantly functorial under arbitrary morphisms
\cite[Theorem I.6.8]{Kollarrat}. Therefore, the morphism $W\to X$
gives a morphism $e$ from $C$ to the Chow variety
of 1-cycles on $X$.

For each $k$-point $c$ in $C$ (possibly in $Z$),
$e(c)$ is an effective 1-cycle in $X$ whose support $S$ is the image
under $G$ of 
the inverse image of $c$ in $W$. Let us describe this image using
power series.
The completed local ring of $C$ at $c$ is isomorphic
to the power series ring $k[[u]]$. So the curve $f\colon C\to X=\P^r$
near $c$ is given by some power series $[z_0(u),\ldots,z_r(u)]$
with $z_i(u)\in k[[u]]$, not all zero. 
Every point in the inverse image of $c$ in $W$ is contained
in some curve in $W$ that meets the open set $T\times (C-Z)$ in $W$.
After completion,
this curve determines a finite extension $F$ of the field $k((u))$,
along with an $F$-point of $T\times (C-Z)$ over the given
$k((u))$-point of $C$. Therefore,
the support $S$ of the limit 1-cycle $e(c)$,
viewed as a subset of $X(\overline{k})$,
is the set of all $\overline{k}$-points in $X$
that can be written as
$$p=\lim_{u\to 0} [g(u)^{a_0}z_0(u),\ldots,g(u)^{a_r}z_r(u)]$$
for some $g$ in the algebraic closure $\overline{k((u))}$, $g\neq 0$.

This limit point depends mainly
on the rational number $b:=\ord_u(g)$.
The situation is described by the Newton
polygon of the pairs $(a_i,\ord_u(z_i))$ in $\Z\times (\Z\cup{\infty})$,
as in Figure \ref{newton}. (Here $\ord_u(z_i)=\infty$
if $z_i(u)$ is identically zero.)
\begin{figure}\centering
\begin{tikzpicture}[xscale=1, yscale=1]
\draw [<->] (0,4) -- (0,0) -- (6,0);
\node [below right] at (3,0) {$a_i$};
\node [left] at (0,3) {$\ord_u(z_i)$};
\draw [thick] (1,4) -- (1,2.5);
\draw [thick] (1,2.5) -- (2,1.5);
\draw [thick] (2,1.5) -- (4,1);
\draw [thick] (4,1) -- (5,1.5);
\draw [thick] (5,1.5) -- (5,4);
\draw[fill] (1,2.5) circle [radius=1pt];
\draw[fill] (2,1.5) circle [radius=1pt];
\draw[fill] (2.5,2.5) circle [radius=1pt];
\draw[fill] (3.2,1.2) circle [radius=1pt];
\draw[fill] (4,1) circle [radius=1pt];
\draw[fill] (4.5,3.0) circle [radius=1pt];
\draw[fill] (5,1.5) circle [radius=1pt];
\end{tikzpicture}
\caption{Newton polygon of the pairs $(a_i,\ord_u(z_i))$}
\label{newton}
\end{figure}
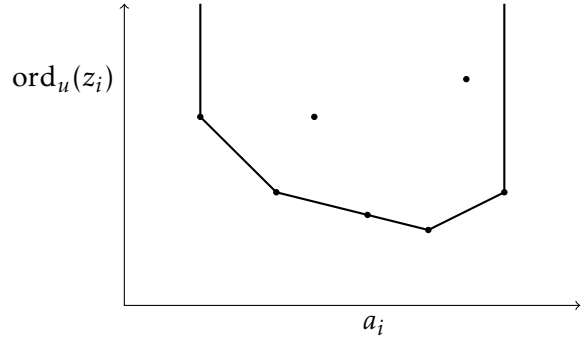

Namely, let $I$ be the set of numbers $i\in\{0,\ldots,r\}$ such that
$ba_i+\ord_u(z_i)$ reaches its minimum value. (That is, let $l$ be
the unique line of slope $-b$ that meets the Newton polygon but not
the region above it; then $I$ corresponds to the points
$(a_i,\ord_u(z_i))$ that lie on $l$.)
Then we compute that the limit point $p$ defined above has
all coordinates zero except the $i$th coordinate for elements $i\in I$.
Replacing $g$ by another function with the same value of $b$ (that is,
multiplying $g$ by a unit $h(u)$) just replaces $p$ by $h(0)(p)$, another point
in the same $T$-orbit as $p$.

For all but finitely many rational numbers $b$,
the limit point $p$ above belongs to a set $\{x_0,x_1,\ldots,x_n\}$
of $T$-fixed points, these being
indexed by the vertices of the Newton polygon. (For these values of $b$,
all nonzero coordinates $i$ in the set $I$ above have the same weight $a_i$,
which means that $p$ is a $T$-fixed point.) For the remaining $n$
values of $b$,
corresponding to the non-vertical edges of the Newton polygon,
the limit point can be anywhere in the $T$-orbit of a certain point $y_i$
in $X$
with $\lim_{t\to 0}t(y_i)=x_{i-1}$ and $\lim_{t\to \infty}t(y_i)=x_i$.
Here the points $y_1,\ldots,y_n$ (and hence the points $x_0,x_1,\ldots,
x_n$) can be taken to be
$k$-points of $X$, by choosing the function $g\in \overline{k((u))}$
with a given
value of $\ord_u(g)\in\Q$ to lie in a totally ramified extension
of $k((u))$, for example in $k((u^{1/e}))$ for a positive integer $e$.
\end{proof}

\section{Openness}

We now prove that the subset $U$ attracted in $Y$ in Conjecture
\ref{conj:torus} is open in $X$.

\begin{lemma}
\label{lemma:open}
As in Conjecture \ref{conj:torus},
let $X$ be a projective scheme over a field $k$
with an action of $T=\mathbf{G}_m$.
Suppose that there is a $T$-equivariant ample line bundle
on $X$. Let $Y$
be a $T$-invariant closed subset of $X$ 
such that every point $x$ in $X$ with
$\lim_{t \to \infty}(tx)\in Y$ is in $Y$. Suppose
that the fixed point set $Y^T$ is open in $X^T$.
Let $U$ be the subset of points $x$ in $X$
such that $\lim_{t \to 0}(tx)$ is in $Y$. Then
$Y$ is contained in $U$, and $U$ is open in $X$.
\end{lemma}

\begin{proof}
Because $Y$ is a $T$-invariant closed subset of $X$,
$Y$ is contained in $U$.

Clearly $U$ is a constructible subset of $X$; in Drinfeld's notation
from section \ref{section:main} above, $U$ is the image in $X$
of some connected components of $X^+$, those whose limit
as $t\to 0$ is in $Y^T$. (By assumption, $Y^T$ is a union
of some connected components of $X^T$.) It suffices to prove
that $U$ is open in $X$ after replacing $k$ by its
algebraic closure. If $U$ is not open
in $X$, then there is a morphism $f$ from a smooth curve $C$ to $X$
with a $k$-point $c\in C$ such that $f(c)\in U$ and $f(d)\not\in U$
for all $d\neq c$ in $C$. Thus $\lim_{t\to 0}t(f(d))\in X^T-Y^T$
for $d\neq c$, whereas $\lim_{t\to 0}t(f(c))\in Y^T$.

By Proposition \ref{broken}, the limit of the $T$-orbit closures
of the points $f(d)$ as $d$ approaches $c$ is a broken trajectory
containing $f(c)$. By what we have said about $f(d)$,
this broken trajectory ends (in the $t\to 0$ direction)
at a point in $X^T-Y^T$. But this broken trajectory also contains
$f(c)$ and hence the point $\lim_{t\to 0}t(f(c))$.
Therefore, there is a broken
trajectory from $\lim_{t\to 0}t(f(c))\in Y^T$ down (in the $t\to 0$
direction) to a point in $X^T-Y^T$. This contradicts our assumption
on $Y$, namely that $Y$ is a $T$-invariant closed subset
such that every point $x$ in $X$ with
$\lim_{t \to \infty}(tx)\in Y$ is in $Y$. We have shown
that $U$ is open in $X$.
\end{proof}

\section{Proof of Theorem \ref{complex}}
\label{section:proof}

\begin{proof}
To recall the assumptions: we have a projective scheme $X$ over $\C$
with an action of $T=\C^*$,
and there is a $T$-equivariant ample line bundle
on $X$. We have
a $T$-invariant closed subscheme $Y$ of $X$ 
such that every point $x$ in $X$ with
$\lim_{t \to \infty}(tx)\in Y$ is in $Y$, and
the fixed point set $Y^T$ is open in $X^T$.
Let $U$ be the subset of points $x$ in $X$
such that $\lim_{t \to 0}(tx)$ is in $Y$; then $U$ is Zariski open in $X$,
and $Y$ is contained in $U$, by Lemma \ref{lemma:open}.
We want to show that
the inclusion $Y\to U$ is a homotopy equivalence in the classical
topology.

\begin{lemma}
\label{zerolimit}
Let $q_1,q_2,\ldots$ be a sequence of complex points in $X$
that converge to a $T$-fixed point $w$. Let $t_1,t_2,\ldots$
be a sequence in $\C^*$ that converges to zero in $\C$.
Then any limit point of the sequence $t_i(q_i)$ in $X$ lies in a broken
trajectory ``below $w$''. That is, such a limit point belongs to the union
of the $T$-orbits of some points
$y_1,\ldots,y_n$ in $X$ and some $T$-fixed points
$x_0,\ldots,x_n=w$ such that $\lim_{t\to 0}t(y_i)=x_{i-1}$
and $\lim_{t\to \infty}t(y_i)=x_i$ for each $1\leq i\leq n$.
\end{lemma}

\begin{proof}
We largely follow the proof of Proposition \ref{broken}.
Choose a $T$-equivariant embedding of $X$ into $P(V)$
for some representation $V$ of $T$. We can write the action
of $T$ on $\P^r=P(V)$ by $t([z_0,\ldots,z_r])=[t^{a_0}z_0,
\ldots,t^{a_r}z_r]$ with $a_0\leq\cdots\leq a_r$. After passing
to a subsequence, we can assume that the points $q_1,q_2,\ldots$
all have the same lowest weight $a_j$ of a nonzero coefficient.
On the locally closed subset $K$ in $X$ of points with this lowest
weight, the $T$-action $T\times K\to K$
extends to a morphism $f\colon \mathbf{A}^1\times K\arrow \overline{K}$,
by inspection. Here $\overline{K}$ denotes the closure of $K$
in $X$.

By assumption,
the points $(t_i,q_i)$ in $A_1\times K$ converge to the point
$(0,w)$ in $\mathbf{A}^1\times \overline{K}$. The rational map
$f\colon \mathbf{A}^1\times \overline{K}\dashrightarrow \overline{K}$ becomes
a morphism after some blow-up $M\to \mathbf{A}^1\times \overline{K}$
that is an isomorphism over the complement of $0\times(\overline{K}-K)$.
So any limit point of the sequence $t_i(q_i)$ in $X$ is equal
to $f(m)$ for some point $m$ in $M$ over $(0,w)\in \mathbf{A}^1\times\overline{K}$.
In particular, we can choose
a smooth algebraic curve with a morphism to $M$
that goes through $m$ and meets
the open set $T\times K$.

Thus, by considering the completion of this curve at the point
that maps to $m$, we have power series $g(u)\neq 0\in\C[[u]]$ and
$z(u)\in X(\C((u)))$ such that $g(0)=0$, $\lim_{u\to 0}z(u)=w$,
and $\lim_{u\to 0}(g(u))(z(u))$
is the given limit point in $X$. The proof of Proposition
\ref{broken} showed that the limit of the closures of $T$-orbits of
$z(u)$ as $u$ approaches 0 is a broken trajectory in $X$, which clearly
contains $w$ as one of the $T$-fixed points $x_0,\ldots, x_n$,
say $w=x_j$. Moreover, since $g(0)=0$ (so that $b:=\ord_u(g)>0$),
the explicit calculation of $\lim_{u\to 0}(g(u))(z(u))$ in $\P^r$
shows that this limit point is ``below $w$'', that is, in the union
of $x_0,\ldots,x_j=w$ and the $T$-orbits that connect them.
\end{proof}

We continue the proof of Theorem \ref{complex}.
By the triangulation of real semialgebraic sets,
there is a triangulation of $X$ with $Y$ as a subcomplex
\cite[Section~1]{Hironaka1975}. Therefore,
$Y$ has arbitrarily small simplicial regular neighborhoods $N$ in $X$,
and for these the inclusion $Y\to N$ is a homotopy equivalence
\cite[Chapter~3]{RS}.
Let $N$ be a (compact) regular neighborhood
of $Y$ contained in $U$.

Consider the submonoid $(0,1]$ of $T=\C^*$.
It would be convenient to have $(0,1]\cdot N\subset N$,
but it is not obvious that we can arrange that.
Instead, we argue as follows. I claim that each point $w\in Y^T$
has a neighborhood $N_1$ in $U$ such that $t(N_1)\subset N$
for all $t\in (0,1]$. If not, then there would be a sequence
$q_i$ in $U$ converging to $w$ such that for each positive integer $j$,
$(0,1]\cdot q_j$ is not contained in $N$. 
So there is a sequence
$t_i\in (0,1]$ such that $t_i(q_i)$ is not in $N$. The sequence
$t_i$ must converge to zero; otherwise, a subsequence
of $t_i(q_i)$ would converge to the $T$-fixed point $w$ in $Y$
(and hence infinitely many of those points would be in $N$).

After passing
to subsequences, we can assume that $t_i(q_i)$ converges to a point $v$
in $X-\inter(N)$, hence not in $Y$. By Lemma \ref{zerolimit},
$v$ belongs to the union of some finite chain of $T$-orbits going ``down''
from $w$, meaning the $T$-orbits of some points
$y_1,\ldots,y_n$ in $X$ and some $T$-fixed points
$x_0,\ldots,x_n=w$ such that $\lim_{t\to 0}t(y_i)=x_{i-1}$
and $\lim_{t\to \infty}t(y_i)=x_i$. 
By our assumption that all points $x$ in $X$
with $\lim_{t\to\infty}tx\in Y$ are in $Y$, it follows that $v$ is in $Y$,
a contradiction. Thus we have proved
the claim that each point $w\in Y^T$
has a neighborhood $N_1$ in $U$ such that $t(N_1)\subset N$
for all $t\in (0,1]$.

More generally, for each point $x\in U$ (not just in $Y^T$),
there is a real number $a\in (0,1]$ and
a neighborhood $N_1$ of $x$ in $U$ such that $t(N_1)\subset N$
for all $t\in (0,a]$. That follows from the previous
statement applied to the point $y=\lim_{t\to 0}t(x)\in Y^T$.

Therefore, for every compact subset $K$ of $U$,
there is a real number $a\in (0,1]$ such that
$t(K)\subset N$ for all $t\in (0,a]$. Equivalently,
$K\subset a^{-1}(N)$. In particular, there is a real number $c>1$
such that
the compact neighborhood $N$ of $Y$ is contained in the interior
of $c(N)$. It also follows that $U$ is the union of the subsets
$c^j(N)$ over all $j\geq 0$. 

Since the inclusion $Y\to N$ is a homotopy equivalence, so is the inclusion
$Y\to c^j(N)$ for each integer $j$. Therefore, each of the inclusions
$c^j(N)\to c^{j+1}(N)$ is also a homotopy equivalence.
Since $c^j(N)$ is a closed subset contained
in the interior of $c^{j+1}(N)$, the union of these subsets (namely, $U$)
has the colimit topology. Since this is a filtered colimit,
the colimit $U$ is equivalent to the homotopy colimit, and so
the inclusion $N\to U$ is a homotopy equivalence.
Since the inclusion $Y\to N$ is also a homotopy equivalence,
we conclude that $Y\to U$ is a homotopy equivalence.
\end{proof}

\section{The real case}

\begin{theorem}
\label{real}
Under the assumptions of Conjecture
\ref{conj:torus} with base field $\R$, the inclusion $Y(\R)\to U(\R)$
is a homotopy equivalence.
\end{theorem}

\begin{proof}
This is similar to the complex case (Theorem \ref{complex}).
In particular, Lemma \ref{zerolimit} holds by the same proof
over $\R$ in place of $\C$, using that Proposition \ref{broken}
expresses any limit of $T$-orbits of $\R$-points as the union
of a finite chain of $T$-orbits of $\R$-points.
Given that, the proof of Theorem \ref{complex} applies
verbatim (using a regular neighborhood of $Y(\R)$ inside
$U(\R)$) to show that the inclusion $Y(\R)\to U(\R)$
is a homotopy equivalence.
\end{proof}

\section{$\mathbf{A}^1$-connectedness of the Hilbert scheme}

Hartshorne showed that the Hilbert scheme of projective
space over a field $k$ (of subschemes with a given
Hilbert polynomial) is connected \cite{HartshorneHilbert}.
In particular, $\Hilb_d(\P^n)$ is connected for every $n\geq 1$
and $d\geq 0$. The argument was sharpened by Reeves and Pardue
\cite{Reeves, Pardue}. Reeves and Pardue showed that
for an infinite field $k$, any two $k$-points of $\Hilb_d(\P^n)$
can be connected by a chain of affine lines over $k$.
By Morel's results (Lemma \ref{connected} below),
it follows that $\Hilb_d(\P^n_k)$
is $\mathbf{A}^1$-connected for $k$ infinite.

We now show that $\Hilb_d(\mathbf{A}^n)$ and $\Hilb_d(\mathbf{A}^n,0)$
are $\mathbf{A}^1$-connected over an infinite field $k$. This seems to be harder
for $\Hilb_d(\mathbf{A}^n,0)$, because
(for $d>1$) this space contains no smooth
subschemes of $\mathbf{A}^n$. When $n\geq d$, the $\mathbf{A}^1$-connectedness
of these Hilbert schemes
can be proved using the ideas of \cite{HJNTY},
but here we want the results for all $n$ and $d$.

\begin{theorem}
\label{hilbconnected}
Let $k$ be an infinite field, $n\geq 1$, $d\geq 0$.
Then $\Hilb_d(\mathbf{A}^n)$ is $\mathbf{A}^1$-connected over $k$.
\end{theorem}

\begin{proof}
We use the following result of Morel's:

\begin{lemma}
\label{connected}
Let $X$ be a separated scheme of finite type over a field $k$
such that $X$ has a $k$-point.
Suppose that for every separable finitely generated field
extension $F$ of $k$, any two $F$-points of $X$ can be connected
by a chain of affine lines $\mathbf{A}^1_F\to X_F$. Then $X$
is $\mathbf{A}^1$-connected.
\end{lemma}

\begin{proof}
For $m\geq 0$,
Morel showed that an $\mathbf{A}^1$-local pointed simplicial Nisnevich sheaf $X$
over $k$ is $m$-connected if and only if the fiber $X(F)$ is $m$-connected
for every separable finitely generated field extension $F$ of $k$
\cite[Lemma 6.1.3]{Morel-connectivity}. Also, for a simplicial sheaf $X$,
$\pi_0(X)\to\pi_0^{\mathbf{A}^1}(X)$ is a surjection of Nisnevich sheaves
\cite[Section 2, Corollary 3.22]{MV}. In particular, for a separated
scheme $X$ of finite type
over $k$, $X(F)\to \pi_0^{\mathbf{A}^1}(X)(F)$ is surjective for every
separable finitely generated field extension $F$ over $k$. This implies
the lemma.
\end{proof}

By Lemma \ref{connected},
it suffices to show that for an infinite field $k$,
any two $k$-points of $U:=\Hilb_d(\mathbf{A}^n)$ can be connected
by a chain of affine lines $\mathbf{A}^1_k\to U_k$.
So let $S$ be any $k$-point of $U$. That is,
$S$ is a closed subscheme of $\mathbf{A}^n$ over $k$ of dimension zero
and degree $d$. We use a ``Gr\"obner degeneration'',
as follows. Let $c$ be a large positive integer, and consider
the action of $T:=\mathbf{G}_m$ on $\mathbf{A}^n$ by
$$t(x_1,\ldots,x_n)=(t^cx_1,t^{c^2}x_2,\ldots,t^{c^n}x_n).$$
Then $S_2:=\lim_{t\to 0}t(S)$ exists in $U$.  It is a closed
subscheme supported at the origin in $\mathbf{A}^n$,
and it is fixed by this $T$-action. That is, the defining
ideal $I$ of $S_2$ as a subscheme of $\mathbf{A}^n_k$
is homogeneous with respect to the weights
$(c,c^2,\ldots,c^n)$ on $x_1,\ldots,x_n$. Taking $c$ big enough
compared to $d$ and $n$, it follows that $I$ is generated by monomials.
By construction, we can connect $S$ to $S_2$ by an affine line over $k$.

Since $S_2$ has dimension 0 and is defined by monomials, it is smoothable,
using Hartshorne's {\it proof by distraction}; a specific reference
is \cite[Proposition 4.15]{Cartwright}. We need the more precise
information given by the proof, as follows.
Let $I=(x^{M_1},\ldots,x^{M_r})$ be the minimal set of monomial
generators for the ideal $I$. We use
multi-index notation, so $x^{M_i}=\prod_{j=1}^n x_j^{M_{ij}}$.
Consider the following flat family
of ideals in $k[x_1,\ldots,x_n]$
parametrized by affine space $A^d$: for a point
$(a_0,\ldots,a_{d-1})$ in $A^d$, take the ideal $J_a$ in $k[x_1,\ldots,x_n]$
generated by the elements
$$f_i:=\prod_{j=1}^n (x_j-a_0)(x_j-a_1)\cdots (x_j-a_{M_{ij}-1}).$$
The initial ideal of $J_a$ (with respect to any monomial order
compatible with the grading, say the graded reverse lexicographic order)
is $I$; so we have a flat family.
This defines a morphism $A^d\to \Hilb_d(\mathbf{A}^n)$ over $k$,
with the origin
mapping to the given monomial scheme $S_2$. When $a_0,\ldots,a_{d-1}$
are distinct elements of $k$, the subscheme $Z_a$ of $\mathbf{A}^n$ defined by
$J_a$ contains $d$ distinct $k$-points: namely, for each of the $d$
monomials $x^L$ not in $I$, $Z_a$ contains the $k$-point
$(a_{L_1},\ldots,a_{L_n})$. Since the scheme $Z_a$ has degree $d$,
it must be smooth over $k$, equal to those $d$ $k$-points in $\mathbf{A}^n$.

Since $k$ is infinite, it follows that we can connect $S_2$ by an affine
line in $\Hilb_d(\mathbf{A}^n)$ to a scheme $S_3$ which consists
of $d$ distinct $k$-points in $\mathbf{A}^n$. If $n\geq 2$, since the condition
for two points to be equal in $\mathbf{A}^n$ has codimension at least 2,
it is easy to connect $S_3$ by a chain of affine lines over $k$
to a fixed arrangement $S_4$ of $d$ distinct $k$-points in $\mathbf{A}^n$.
Thus $\Hilb_d(\mathbf{A}^n)$ is $\mathbf{A}^1$-connected when $n\geq 2$.
It is also $\mathbf{A}^1$-connected
when $n=1$, since $\Hilb_d(\mathbf{A}^1)\cong A^d$.
\end{proof}

\begin{theorem}
\label{hilbzeroconnected}
Let $k$ be an infinite field, $n\geq 1$, $d\geq 0$.
Then $\Hilb_d(\mathbf{A}^n,0)$ is $\mathbf{A}^1$-connected over $k$.
\end{theorem}

\begin{proof}
By Lemma \ref{connected}, it suffices
to show that for every infinite field $k$, any two
$k$-points of $Y:=\Hilb_d(\mathbf{A}^n,0)$ can be connected by a chain
of affine lines over $k$. For lack of a direct proof,
we will reduce this to Theorem \ref{hilbconnected}.

Let $X=\Hilb_d(\P^n)$,
$U=\Hilb_d(\mathbf{A}^n)$, and $T=\mathbf{G}_m$. Consider the action of $T$ on $X$
coming from the action of $T$ by scaling on $\mathbf{A}^n$. Then $Y$ is a $T$-invariant
closed subset of $U$, and
$\lim_{t\to 0}t(x)$ exists in $Y$ for each point $x$ in $U$.
Clearly we can connect any $k$-point $x$ in $Y$ to this limit point
by an affine line in $Y$, and the limit point is fixed by $T$.
So it suffices to show that any two $k$-points $p,q$ in $Y^T$
can be connected by a chain of affine lines in $Y$.

We know by the proof of Theorem \ref{hilbconnected}
that $p$ and $q$ can be connected by a chain of affine
lines in $U$. So it suffices to show that for any morphism
$f\colon \mathbf{A}^1\to U$ over $k$, we can connect $\lim_{t\to 0}t(f(0))$
to $\lim_{t\to 0}t(f(1))$ by a chain of affine lines in $Y$.

Composing $f$ with the action of $T$ on $U$ gives a morphism
$T\times \mathbf{A}^1\to U$ over $k$, which can be viewed as a rational map
$\P^1\times \P^1\dashrightarrow X$ over $k$.
Since $X$ is proper over $k$,
this map becomes a morphism after blowing up the domain finitely many times
at closed points. It follows that $g(s):=\lim_{t\to 0}t(f(s))$
defines a morphism $g\colon \mathbf{A}^1-Z\to Y$ for some 0-dimensional closed subset
$Z$ of $\mathbf{A}^1$. Since $Y$ is proper over $k$, $g$
extends to a morphism $g\colon \mathbf{A}^1\to Y$. As a result,
for any two $k$-points $s_1,s_2$ in $\mathbf{A}^1-Z$, $\lim_{t\to 0}t(f(s_1))$
and $\lim_{t\to 0}t(f(s_2))$ can be connected
by an affine line in $Y$.

There remains the case where $0$ or 1 is in $Z$. It suffices
to show that
for any $k$-point $s_0$ in $Z$ (which will be 0 or 1 for us),
the point $z_0:=\lim_{t\to 0}t(f(s_0))$
can be connected
by a chain of affine lines in $Y$ to $g(s_0)$.

By Proposition \ref{broken},
the $T$-orbits of the points $f(s)$ (for $s\in \mathbf{A}^1-Z)$
converge as $s$ approaches $s_0$ to a ``broken trajectory'' containing
$f(s_0)$. This means the union of $T$-orbits
of points
$y_1,\ldots,y_n$ in $X(k)$ that connect $T$-fixed points
$x_0,\ldots,x_n$ in $X(k)$, in the sense that
$\lim_{t\to 0}t(y_i)=x_{i-1}$
and $\lim_{t\to \infty}t(y_i)=x_i$.

Both the $k$-point $z_0=\lim_{t\to 0}t(f(s_0))$ and the $k$-point $g(s_0)$
lie in this union of $T$-orbit closures in $X$, and both are in
the closed subset $Y$. We know that every point $x$ in $X$
with $\lim_{t\to\infty}(tx)\in Y$ is in $Y$. Therefore,
all the orbit closures that connect $z_0$ to $g(s_0)$ are in $Y$.
So these two points can be connected by affine lines over $k$ in $Y$,
as we want.
\end{proof}

\section{Conservativity for the motivic stable homotopy category}

Extending one of Tom Bachmann's results, we prove the following
conservativity theorem, relating the motivic stable homotopy category
with the derived category of motives
along with real realizations.
Thanks for Bachmann for his suggestions.
This result will be used in the proof of Theorem \ref{smooth}.

\begin{theorem}
\label{conservativity}
Let $k$ be a finitely generated field of characteristic zero
$($that is, a finitely generated extension field of $\Q)$.
Let $A$ be a compact object in $SH(k)$ such that
$M(A)=0$ in $DM(k)$ and for every embedding of $k$ into $\R$,
$H_*(A(\R),\Z[1/2])=0$. Then $A=0$.
\end{theorem}

If $k$ has no real embedding,
Theorem \ref{conservativity} just says
that $M(A)=0$ in $DM(k)$ implies $A=0$ in $SH(k)$.

\begin{proof}
Bachmann showed (in particular)
that if $A$ is a compact object in $SH(k)$
such that $M(A)=0$ in $DM(k)$ and for every $\sigma$
in the space $\Sper(k)$ of orderings of $k$,
$M_{\sigma}[1/2](A)=0$ in $D(\Z[1/2])$, then $A=0$
\cite[Theorem 33]{Bachmannconservativity}.
When $\sigma$
comes from an embedding of $k$ into $\R$, $M_{\sigma}[1/2](A)$
is the complex that computes the singular homology of the corresponding
real realization of $A$, $H_*(A(\R),\Z[1/2])$
\cite[Remark 1]{Bachmannconservativity}. It remains to show
that we only need to consider orderings that
come from real embeddings of $k$.

We use the following property of the space
$\Sper(k)$ of orderings of $k$ \cite[Lemma 1.6]{FHV}.
The topology on $\Sper(k)$ is defined by taking the sets
$\{\sigma: a>_{\sigma}0\}$ for $a\in k$ as a sub-basis for the topology.
This makes $\Sper(k)$ into a compact Hausdorff totally
disconnected space.

\begin{lemma}
Let $k$ be a finitely generated field of characteristic zero.
Then the set of archimedean orderings of $k$ is dense
in the topological space $\Sper(k)$ of orderings of $k$,
and every archimedean ordering comes from an embedding
of $k$ into $\R$.
\end{lemma}

Given that, we are done if we can show that the support
in $X:=\Sper(k)$ of a compact object in $SH(k)$ is open as well as closed.
This is related to the general fact that for a tensor triangulated
category $K$, the support of an object of $K$ in the Balmer spectrum
$\Spc(K)$ is closed {\it and }its complement is quasi-compact
\cite[Proposition 2.14]{Balmer}. However,
we will argue more directly.

We use that the functors $M_{\sigma}$
come from a functor from $SH(k)$ to the derived category of sheaves
$D(X,\Z[1/2])$,
which takes compact objects to compact objects. Indeed, by
\cite[Lemma 21]{Bachmannconservativity}, the functor
from $SH(k)$ to Witt motives $DM_W(k,\Z[1/2])$ is monoidal;
so it takes rigid objects
to rigid objects, and the rigid objects
coincide with the compact objects in these categories. (Some people
say ``strongly dualizable'' rather than ``rigid''.)
Furthermore,
$DM_W(k,\Z[1/2])$ is equivalent to $D(X,\Z[1/2])$
\cite[Lemma 26 and proof of Theorem 30]{Bachmannconservativity}.

A compact object in $D(X,\Z[1/2])$ is a perfect complex;
that is, it is locally
isomorphic to a bounded complex of finitely generated
projective $\Z[1/2]$-modules. 
(Indeed, since $X$ is compact, Hausdorff,
and totally disconnected, every open subset of $X$ is a union
of clopen subsets (or equivalently, quasi-compact open subsets).
It follows that every compact object
in $D(X,\Z[1/2])$ is a summand in $D(X,\Z[1/2])$ of a bounded
complex of sheaves which are finite direct sums of sheaves
of the form $j_!(\Z[1/2]_U)$, with $j\colon U\inj X$
the inclusion of a quasi-compact open subset \cite[Lemma 094C]{Stacks}.
Clearly such a summand is a perfect complex of $\Z[1/2]$-modules
on $X$.)

Because sections of the sheaf $\Z[1/2]$ on $X=\Sper(k)$ are locally
constant, the support of a perfect complex on $X$
is open as well as closed.
\end{proof}

\section{$\mathbf{G}_m$-actions on smooth varieties
and motivic homotopy theory}

We now consider Conjecture~\ref{conj:torus} in the special case where $U$
is smooth. (One example where this applies is the inclusion
from $\Hilb_d(A^2,0)$ to $\Hilb_d(A^2)$.) For $U$ smooth,
we show that the inclusion $Y\to U$ becomes an $\mathbf{A}^1$-homotopy
equivalence after suspending by $\mathbf{S}^{3,1}=\mathbf{S}^2\wedge \mathbf{G}_m$.
It follows that $Y$ and $U$ have many invariants
in common, such as motivic homology and cohomology, $l$-adic
cohomology, and so on. On the other hand, it remains open
whether the Nisnevich sheaf $\pi_0^{\mathbf{A}^1}$ is the same for $Y$
and $U$, and likewise for $\pi_1^{\mathbf{A}^1}$. At least for $\pi_0^{\mathbf{A}^1}$,
one might hope to imitate the proof of Theorem \ref{hilbzeroconnected}.

\begin{theorem}
\label{smooth}
Under the assumptions of Conjecture
\ref{conj:torus} with base field $k$ of characteristic zero,
and assuming that $U$ is smooth over $k$,
the inclusion $Y\to U$ becomes an $\mathbf{A}^1$-homotopy equivalence
after suspending by $\mathbf{S}^{3,1}=\mathbf{S}^2\wedge \mathbf{G}_m$.
\end{theorem}

\begin{proof}
We can assume that $U$ is connected, by arguing separately
for each connected component of $U$.
Next, by equivariant resolution of singularities (using that $k$
has characteristic zero), we can assume that $X$
(as well as $U$) is smooth over $k$, while still having a $T$-action
\cite[Proposition 3.9.1]{Kollarres}.

We first show that the inclusion $Y\to U$ induces an isomorphism
in the derived category of motives $DM(k)$, $M(Y)\to M(U)$.
Namely, since $X$ is smooth over $k$, we have the Bia\l ynicki-Birula
decomposition, as follows. The fixed point set $X^T$
is smooth over $k$. Write $Z_1,\ldots,Z_m$ for the connected
components of $X^T$.
For each $i$, let $Z_i^+=\{x\in X: \lim_{t\to 0}tx\in Z_i\}$
and $Z_i^-=\{x\in X: \lim_{t\to \infty}tx\in Z_i\}$ be the stable
and unstable manifolds of $Z_i$. Then the action of $T$ gives
morphisms $Z_i^+\to Z_i$ and $Z_i^-\to Z_i$ which are affine-space
bundles \cite{BB}.

Karpenko showed that this geometric decomposition
gives a direct-sum
decomposition of Chow motives over $k$ \cite[Theorem 6.5]{Karpenko},
\cite[Theorem 3.5]{Brosnan}:
$$M(X)\cong \oplus_{i=1}^m M(Z_i)\{a_i\},$$
where $a_i:=\dim(Z_i^+)-\dim(Z_i)$. (Here $\Z\{1\}$ denotes the Lefschetz
motive, with $M(\P^1)=\Z\{0\}\oplus\Z\{1\}$.)
This implies another decomposition $M(X)\cong
\oplus_{i=1}^m M(Z_i)\{b_i\}$,
where $b_i:=\dim(Z_i^-)-\dim(Z_i)$, by inverting the $T$-action on $X$.
Here
$$a_i+\dim(Z_i)+b_i=n,$$
by considering the action on $T$
on the tangent space to $X$ at a point of $Z_i$.

The category of Chow motives is a full subcategory
of the derived category of motives, $DM(k)$: the thick subcategory
generated by smooth projective schemes over $k$ tensored with 
$\Z\{a\}=\Z(a)[2a]$
for integers $a$ \cite{Voevodskytri}. 
Every scheme $X$ of finite type over $k$
has a motive $M(X)$ and a compactly supported motive $M^c(X)$
in $DM(k)$. We can assume that $Z_1,\ldots,Z_m$ are ordered
in such a way that the closure of $Z_i^-$ is contained
in $X_i:=\cup_{j\leq i}Z_j^-$. Karpenko's argument shows that
the exact triangle
$$M^c(X_{i-1})\to M^c(X_i)\to M^c(Z_i^-)\cong M(Z_i)\{b_i\}$$
in $DM(k)$ is split \cite[Theorem 6.5, part (a)]{Karpenko}.
(Indeed, his splitting on Chow groups is defined by an element
of $CH_{\dim(Z_i)+b_i}(Z_i\times X_i)$, and that is precisely
$\Hom(M(Z_i)\{b_i\},M^c(X_i))$ since $Z_i$ is smooth and proper over $k$.)
In particular, it follows that
$$M(X_j)\cong \oplus_{i=1}^j M(Z_i)\{b_i\}$$
for each $1\leq j\leq m$, and its open complement
$X-X_j$ satisfies
$$M^c(X-X_j)\cong \oplus_{i=j+1}^m M(Z_i)\{b_i\}$$
(Here $X_j$ need not be smooth,
but it is proper over $k$, and so its motive $M(X_j)$
is the same as its compactly supported motive $M^c(X_j)$.)

In the notation of Conjecture \ref{conj:torus},
we can assume that the closed subset $Y$ of $X$ is equal to $X_r$
for some $r\leq m$.
So $M(Y)\cong \oplus_{i=1}^r M(Z_i)\{b_i\}$.
Likewise, the open subset $U$ is the union of the subsets
$Z_i^+$ with $i\leq r$. By the splitting in $DM(k)$ above, applied
to the inverse action of $T$ on $X$,
we have
$$M^c(U)\cong \oplus_{i=1}^r M(Z_i)\{a_i\}.$$
Since $U$ is smooth of dimension $n$ over $k$, it follows
that
\begin{align*}
M(U)&\cong M^c(U)^*\{n\}\\
&\cong \oplus_{i=1}^r M(Z_i)^*\{n-a_i\}\\
&\cong \oplus_{i=1}^r M(Z_i)\{n-a_i-\dim(Z_i)\}\\
&\cong \oplus_{i=1}^r M(Z_i)\{b_i\}.
\end{align*}

Thus $M(Y)$ is isomorphic to $M(U)$ in $DM(k)$. More precisely,
the inclusion $Y\to U$ induces an isomorphism $M(Y)\to M(U)$.
To see this, one checks from Karpenko's construction
of the splittings that for $i,j\in \{1,\ldots,r\}$,
the composition $M(Z_i)\{b_i\}\to M(Y)\to M(U)\to M(Z_j)\{b_j\}$
is the identity for $i=j$ and zero if $i<j$.

The schemes $Y,U,X$ with $T$-action are defined over 
some finitely generated subfield of $k$.
So we can assume
that the field $k$ is finitely generated over $\Q$.
Apply Theorem \ref{conservativity}
to the cofiber $A=\Sigma^{\infty}(U/Y)$ in $SH(k)$.
We showed above that the motive of $A$ in $DM(k)$ is zero.
Also, for every real embedding of $k$, the real realization
of $A$ is zero in the stable homotopy category, by Theorem
\ref{real}. (It may be that $k$ has no real embedding.)
Therefore, $A=0$ in $SH(k)$.
That is, the inclusion $Y\to U$ induces an isomorphism
in $SH(k)$.

Again using that $k$ has characteristic
zero, Bachmann showed that the $\P^1$-suspension functor
$$Q=\Sigma^{\infty}_{\P^1}\colon H(k)_* \to SH(k)$$
is conservative on $\mathbf{A}^1$-simply connected spaces which can be written as homotopy colimits
of spaces $X_+\wedge \mathbf{G}_m$ with $X\in \Sm_k$
\cite[Theorem 1.3]{BachmannP1}. 

The $\mathbf{S}^2$-suspension of every space in $H(k)_*$
is $\mathbf{A}^1$-simply connected. The map $\mathbf{S}^2\wedge \mathbf{G}_m\wedge Y_+\to \mathbf{S}^2\wedge \mathbf{G}_m\wedge U_+$
is therefore a pointed $\mathbf{A}^1$-homotopy equivalence.
\end{proof}

Theorem \ref{smooth} can be slightly strengthened
if in addition $Y$ and $U$
are $\mathbf{A}^1$-connected. In that case, their $\mathbf{S}^1$-suspensions
are $\mathbf{A}^1$-simply connected, and so $\P^1\wedge Y_+\to \P^1\wedge U_+$
is a pointed $\mathbf{A}^1$-homotopy equivalence, using that $\P^1=\mathbf{S}^{2,1}
=\mathbf{S}^1\wedge \mathbf{G}_m$.

\end{document}